\documentclass[10pt,a4paper, reqno]{amsart}
\usepackage{graphicx}
\usepackage{tikz}

\usepackage{amsmath}%
\usepackage{amssymb,indentfirst,calc,mathrsfs,euscript,bbm}%
\usepackage{dsfont}

\usepackage{setspace}%
\usepackage[reals]{layout}%
\usepackage{xr}%
\usepackage{amscd}%

\usepackage{comment}
\usepackage{epstopdf}
\usepackage{setspace}
\doublespace

\def\rit{\hbox{\it I\hskip -2pt  R}}

\def\A{{\bf A}}

\def\R{{\bf R}}

\def\bX{{\bf X}}
\def\bY{{\bf Y}}

\def\bx{{\bf x}}
\def\by{{\bf y}}

\newtheorem{theorem}{Theorem}
\newtheorem{proposition}[theorem]{Proposition}
\newtheorem{lem}[theorem]{Lemma}
\newtheorem{corollary}[theorem]{Corollary}

\newtheorem{defi}[theorem]{Definition}

\newtheorem{rem}[theorem]{Remark}
\newtheorem{example}[theorem]{Example}

\DeclareMathOperator{\med}{med}

\begin{document}

\title[ZSSG:LOT]{Zero-sum stochastic games: limit optimal trajectories}
\author{
Sylvain Sorin}

\address{\textbf{Sylvain Sorin},
Sorbonne Universit\'e, UPMC  Paris 06, Institut de Math\'ematiques de Jussieu-Paris Rive Gauche, UMR 7586, CNRS,  F-75005, Paris, France }

\email{sylvain.sorin@imj-prg.fr\newline
https://webusers.imj-prg.fr/sylvain.sorin}%

\thanks{Some of the results  of this paper were presented  in ``Atelier Franco-Chilien: Dynamiques, optimisation et apprentissage" Valparaiso, November 2010 and a  preliminary  version of this paper was given  at the Game Theory Conference in Stony Brook,  July 2012. This research  was supported by grant  PGMO
0294-01 (France)}

\author{Guillaume Vigeral } 

\address{\textbf{Guillaume Vigeral (corresponding author)}\\
Universit\'e Paris-Dauphine, PSL Research University, CNRS, CEREMADE, Place du Mar{\'e}chal De Lattre de Tassigny. 75775 Paris cedex 16, France}

\email{
vigeral@ceremade.dauphine.fr\newline
http://www.ceremade.dauphine.fr/~vigeral/indexenglish.html
}

\date{\today}

%

\bibliographystyle{apalike}

\begin{abstract}
We consider zero sum  stochastic games. For every discount factor $\lambda$,  a time normalization allows to represent  the game as being  played  on  the interval $[0,1]$. We introduce  the trajectories  of    cumulated  expected payoff and  of cumulated occupation measure  up to  time $t\in[0,1]$, under $\varepsilon$-optimal strategies. A limit   optimal trajectory  is defined as an accumulation point as the discount factor tends to 0. We study  existence, uniqueness and characterization of these limit optimal trajectories for absorbing games.

\end{abstract}

\maketitle


\section{Introduction}
The analysis of  two person zero sum  repeated games in discrete time may be performed along two lines:\\
1) {\it asymptotic approach}: to each probability  distribution $\theta$ on the set of  stages ($m= 1, 2, ...$)   one associates  the game  $G_\theta$ where the evaluation of the stream of stage payoffs $\{ g_m\} $ is $\sum_{m=1}^{+\infty} \theta_m g_m$,   and  one denotes  its  value by $v_{\theta} $. Given a  preordered family
$\{ \Theta, \succ\}$ of probability  distributions, one studies whether $v_{\theta} $ converges as $\theta \in \Theta$ ``goes to $\infty$''  according to $\succ$. Typical exemples correspond to $n$-stage games ($ \theta_m = \frac{1}{n}I_{m\leq n}, n \rightarrow  \infty$),
 $\lambda$-discounted games ($\theta_m = \lambda (1- \lambda) ^{m-1},  \lambda  \rightarrow  0$), or  more generally decreasing evaluations  ($\theta_m \geq \theta_{m+1}$)  with $\theta_1   \rightarrow  0$. The game has an asymptotic value  $v^*$ when these limits exist and coincide.\\
 2) {\it uniform approach}: for  each strategy of player 1, one evaluates the amount that can be obtained against any strategy of the opponent in any sufficiently long game for $\{ \Theta, \succ\}$. This  allows to define a $minmax$ and the game  has a uniform value  $v_\infty$ when $minmax= maxmin$.
 
 The second approach is stronger than the first one (the existence of $v_\infty$ implies the existence of $v^*$ and their equality), but there are games with asymptotic value and no uniform value (incomplete information on both sides, Aumann and Maschler \cite{AumannMaschler}, Mertens and Zamir \cite{MZ}; stochastic games with signals on the moves).
 The first approach deals only with families of values  while the second explicitly consider strategies. The main difference is that in the first case the ($\varepsilon$)-optimal strategies of  the players may depend on  the evaluation  represented by  $\theta$.

 We focus here on a class of games where this dependence has a smooth representation.
 Basically in addition to the asymptotic properties of the value one studies  the asymptotic behavior  along the play induced  by ($\varepsilon$)-optimal strategies.\\
 The first step is to normalize the duration  of the game using the evaluation $ \theta$. We  consider each game $G_\theta$ as being played on $[0,1]$, stage $n$ lasting from time $t_{n-1}= \sum_{m<n} \theta_m$ 
 to time $t_{n}= t_{n-1} + \theta_n$ (with $t_0= 0$). \\
 Note that here time $t$ corresponds to the fraction $t$ of the total duration of the game, as evaluated trough $\theta$. 
  In particular given   ($\varepsilon$)-optimal strategies in $G_\theta$   the stream of  expected stage payoffs generate a bounded measurable trajectory on  $[0, 1]$ and one will consider its asymptotic behavior.
  
  The next section introduces the basic definitions and concepts  that allow to describe our results.
  The main proofs are in Section 3. Further examples are in Sections 4 and 5. 
 
%
%
%
%
%
 
 To end this  quick overview let us recall that there are games which do not have an asymptotic value: stochastic games with compact action spaces, Vigeral \cite{V4}; finite stochastic games with signals on the state, Ziliotto \cite{Zi1}; or more generally Sorin and Vigeral, \cite{SV2}.
 
\section{Limit  optimal trajectories}
Let $\Gamma$ be a two-person zero-sum stochastic game with state space $\Omega$, action spaces $I$  and $J$,  stage payoff $g$ and transition $\rho$ from $\Omega \times I \times J$ to $\rit$ (resp. $\Delta(\Omega)$). We assume that $\Omega$ is finite, $I$ and $J$ are compact metric, $g$ and $\rho$ continuous. We keep the same notations for the multilinear  extensions to  $X=\Delta(I)$ and $Y=\Delta(J)$, where as usual $\Delta(A)$ denotes the probabilities  on $A$.

For any pair of stationary strategies $(\bx,\by) \in X^\Omega\times Y^\Omega = \bX \times \bY $, any state $\omega\in\Omega$ and any stage $n$, denote by $c^{\omega,\bx,\by}_n$ the expected payoff at stage $n$ under these stationary strategies, given the initial state  $\omega$, and by $q^{\omega,\bx,\by}_n\in\Delta(\Omega)$ the  corresponding distribution  of the state at stage $n$. Hence $c^{\omega,\bx,\by}_n = \langle q^{\omega,\bx,\by}_n,  \bar g(\bx, \by) \rangle  $ where $\bar g(\bx, \by)$ stands for the vector payoff with component in state $\zeta \in \Omega$  given by  $ g({\zeta}; \bx(\zeta), \by(\zeta)) $.
\begin{defi}$ $ \\
For any $(\bx,\by) \in X^\Omega\times Y^\Omega$, any discount factor $\lambda \in (0,1]$, and any starting state $\omega$, define the function $l^{\omega,\bx,\by}_\lambda:[0,1]\rightarrow \mathds{R}$ by
\[
l^{\omega,\bx,\by}_\lambda (t_n)=\lambda\sum_{i=1}^n (1-\lambda)^{i-1} c^{\omega,\bx,\by}_i
\]
 for $t_n= \lambda\sum_{i=1}^n (1-\lambda)^{i-1}$ and a linear interpolation between these dates  $\{ t_n\}$.
 \end{defi}
Thus $l^{\omega,\bx,\by}_\lambda(t_n)$ corresponds to the expectation of the  accumulated payoff for the $n$ first stages in the discounted game,  or up to time $t_n$ and  
 $l^{\omega,\bx,\by}_\lambda(t)$ to  the same at  the fraction $t$ of the game, both  under $\bx$ and $\by$ in the $\lambda$-discounted game starting from $\omega$.

Let $M(\Omega)$ denote the set of positive measures on $\Omega$. We introduce  similarly the expected accumulated occupation measure at  time $t$ under $\bx$ and $\by$ in the $\lambda$-discounted game starting from $\omega$ as follows:
\begin{defi} \label{dQ}$ $ \\
For any $(\bx,\by) \in X^\Omega\times Y^\Omega$, any discount factor $\lambda$, and any starting state $\omega$, define the function $Q^{\omega,\bx,\by}_\lambda:[0,1]\rightarrow M(\Omega)$ by:
\[
Q^{\omega,\bx,\by}_\lambda(t_n)=\lambda\sum_{i=1}^{n} (1-\lambda)^{i-1} q^{\omega,\bx,\by}_i
\]
and  by a linear interpolation between these dates $\{t_n\}$.
\end{defi}
Note that  for any $t\in [0,1]$, ${Q^{\omega,\bx,\by}_\lambda\left(t\right)} \in  t \,  \Delta(\Omega)$.

Denote by $l^{\bx,\by}_\lambda$ and $Q^{\bx,\by}_\lambda$ the $\Omega$-vectors of functions $l^{\omega,\bx,\by}_\lambda(\cdot)$ and $Q^{\omega,\bx,\by}_\lambda(\cdot)$ respectively.

Limit trajectories for the payoff and occupation measures will be defined as  accumulation points of  $l^{\bx_\lambda,\by_\lambda}_\lambda$ and $Q^{\bx_\lambda,\by_\lambda}_\lambda$ under $\lambda-$dependent   $\varepsilon$-optimal strategies $\bx_\lambda$ and $\by_\lambda$ as $\lambda$ tends to 0. 
\\
More precisely, denote by $\bX^\varepsilon_\lambda$ (resp. $\bY^\varepsilon_\lambda$) the set of $\varepsilon$-optimal stationary strategies in  the $\lambda$-discounted game $\Gamma_\lambda$ (with value $v_\lambda$)  for Player 1 (resp. for Player 2). \\
Then we introduce:
\begin{defi}
$l=(l^\omega :[0,1]\rightarrow \mathds{R})_{\omega\in \Omega}$ is a limit   optimal trajectory  for the  expected accumulated payoff ($LOTP$)  if :
\[
\forall \varepsilon>0,\ \exists \lambda_0>0,\ \forall \lambda<\lambda_0,\ \exists \bx_\lambda \in X^\varepsilon_\lambda,\ \exists \by_\lambda \in Y^\varepsilon_\lambda,\ \forall \omega\in\Omega,\ \forall t\in[0,1],\ \left|l^\omega(t)-l^{\omega,\bx_\lambda, \by_\lambda}_\lambda(t)\right|\leq \varepsilon.
\]

$Q=(Q^\omega: [0,1]\rightarrow M(\Omega))_{\omega\in \Omega}$ is a limit optimal trajectory  for the  expected accumulated occupation measure ($LOTM$) if :
\[
\forall \varepsilon>0,\ \exists \lambda_0>0,\ \forall \lambda<\lambda_0,\ \exists \bx_\lambda \in X^\varepsilon_\lambda,\ \exists \by_\lambda \in Y^\varepsilon_\lambda,\ \forall \omega\in\Omega,\ \forall t\in[0,1],\ \left\|Q^\omega(t)-Q^{\omega, \bx_\lambda,\by_\lambda}_\lambda(t)\right\|\leq \varepsilon.
\]
\end{defi}

Alternate weaker and stronger definitions are as follows: in both cases, if ``$\forall \lambda<\lambda_0$" is replaced by ``for some $\lambda_n$ going to 0", we will speak of a weak $LOT$. If  ``$\exists \bx_\lambda \in \bX^\varepsilon_\lambda,\ \exists \by_\lambda \in \bY^\varepsilon_\lambda$" is replaced by `$\exists \varepsilon'<\varepsilon,\ \forall \bx_\lambda \in \bX^{\varepsilon'}_\lambda,\ \forall  \by_\lambda \in \bY^{\varepsilon'}_\lambda$", we will speak of a strong $LOT$.

\begin{rem}
\end{rem}
\noindent - a weak  $LOT$  always exists by standard arguments of equicontinuity.\\
- if a  $LOTP$  $l$ exists, $v_\lambda$ converges  to $l(1)$.\\
- if a strong $LOT$  exists, it is unique\\
- no strong $LOTM$  exists in general  (just consider a game where payoff is always 0).\\
- if the game has a uniform value and both players use $\varepsilon$-optimal strategies the average  expected payoff is essentially constant along the play.

%
%

A first approach to this topic  concerns 
one player games (or games where one  player controls the transitions), where there is no  finiteness assumption on $\Omega$.  Assume that $v_\lambda$ converges uniformly then there exists a strong LOTP and it is linear  w.r.t. $t$, which means that the expected payoff is constant along the trajectory  (Sorin, Venel and Vigeral  \cite{SVV}).\\
The same article provides an example of a  two player game with finite action and countable state spaces, where LOTP is not unique.

 The main contributions of the current paper are:\\
 For  absorbing games, existence of a  linear LOTP, and existence of a ``geometric"  algebraic LOTM\\
  For   finite absorbing games, existence of a   strong LOTP.\\
 An exemple of a finite game   where LOTM  is not semialgebraic.\\
An example of  compact absorbing game  with  non uniqueness of  LOTP.

Let us mention recent results of Oliu-Barton and Ziliotto  \cite{mobbz}
establishing the existence of linear  strong LOPT for finite  stochastic games and optimal strategies: the class of games is larger and they allow for any kind of optimal strategies. Our results deal with compact action spaces  and  $\varepsilon$-optimal strategies.
 
As a final comment, let us underline  the fact that 
the previous concepts and definitions can be extended to any repeated game,  for any evaluation  and any type of strategies.

 {\section{Absorbing games}

An  absorbing game  $\Gamma$  is defined by two sets of actions $I$ and $J$, 
two stage payoff  functions  $g$, $g^*$ from $%
I\times J$ to $\left[ -1,1\right] $ and a  probability of absorption  $%
p^*$ from $I\times J$ to $\left[ 0,1\right] .$\\
$I$ and $J$ are compact metric sets,  $g, g^*$ and $p^*$ are (jointly) continuous.

The repeated game is played in discrete time as
follows. At stage $t=1,2,...$ (if  absorption has not yet occurred) player 1
chooses $i_{t}\in I$ and, simultaneously, player 2 chooses $j_{t}\in J$:\\
(i) the payoff at stage $t$ is $g\left( i_{t},j_{t}\right) $;\\
(ii) with probability $p^*\left( i_{t},j_{t}\right)$ absorption  is reached
and the payoff in all future stages $s > t$ is $g^*\left( i_{t},j_{t}\right) $;\\
(iii) with probability $p\left( i_{t},j_{t}\right) :=1-p^* \left( i_{t},j_{t}\right)$ the situation is
repeated at stage $t+1$.

Recall that the asymptotic analysis for these games is due to Kohlberg \cite{Kohlberg74} in the case where $I$ and $J$ are finite and Rosenberg and Sorin \cite{RosenbergSorin}
in the current framework. In either case the  value $v_\lambda$  of the discounted game  $\Gamma_\lambda$
 converges to some $v$ as $\lambda$ goes to 0. This does not require any assumption on the  information of the players. In case of full observation of the actions - or of the stage payoff, a uniform value exists, see Mertens and Neyman \cite{MeNe81} in the finite case and Mertens, Neyman and Rosenberg  \cite{MNR09} for compact actions.


Recall that 
 $X=\Delta(I)$ and $Y=\Delta(J)$ are  the sets of probabilities on $I$ and $J$. The 
 functions $g$, $p$ and $p^*$ are bilinearly extended to $X\times Y$.
Let \[
G^*(x,y) := p^*(x,y) {\overline g}^* (x,y) := \int_{I\times J} p^*(i,j) g^*(i,j) x(di) y(dj).
\] ${\overline g}^*(x,y)$ is thus the expected absorbing payoff conditionally to absorption (and is thus only defined for $p^*(x,y)\neq 0$).

\subsection{An auxiliary game} $ $ \\
Consider the two-person zero-sum game $\A$, defined for any
$(x,x',a)\in S=X^2\times\mathbb{R}^+$ and 
$(y,y',b)\in T= Y^2\times\mathbb{R}^+$, 
by the payoff function
\begin{equation}\label{defA}
A(x,x',a,y,y',b)=
\frac{g(x,y)+a\: G^*(x',y)+b \: G^*(x,y')}{1+a\: p^*(x',y)+b \: p^*(x,y')}.
\end{equation}

\subsubsection{General properties} $ $ \\
The following proposition extends to the compact case results  due to Laraki \cite{Laraki2010a} in the finite case (later simplified by  Cardaliaguet, Laraki and Sorin  \cite{CLS}).

\begin{proposition}  \label{LL}$  $ \\
1) The game $\A$ has a value, which is  $v =  \lim v_\lambda$. \\
More precisely
\[
v= \max_{x\in X}\sup_{(x',a)\in X\times \mathbb{R}^+} \inf_{(y,y',b)\in T} A(x,x',a,y,y',b)=\min_{y\in Y}\inf_{(y',b) \in Y\times \mathbb{R}^+} \inf_{(x,x',a)\in S} A(x,x',a,y,y',b).
\]

\noindent
2) Moreover, if
$(x,x',a)$ is $\varepsilon$-optimal in the game $\A$
 then for any
$\lambda$ small enough the stationary strategy
$\hat x_\lambda:=\frac{x+\lambda ax'}{1+\lambda a}$ is
$2\varepsilon$-optimal in $\Gamma_\lambda$.
\end{proposition}
\begin{proof} $ $\\
1) Consider  an accumulation point $w$ of the family $\{ v_\lambda \}$ and let  $\lambda_n \rightarrow 0$ such that $v_{\lambda_n}$ converges to $w$.\\
We will show   that
\begin{equation}\label{dps0}
w \leq  
\sup\limits_{(x,x',a)\in S}\inf\limits_{(y,y',b)\in T}
\frac{g(x,y)+a\: G^*(x',y)+b \: G^*(x,y')}{1+a\: p^*(x',y)+b \: p^*(x,y')}
\end{equation}
A dual argument  proves at the same time that the family $\{ v_\lambda \}$ converges and that the auxiliary game $\A$ has a value.\\
Let $r_{\lambda} (x,y)$ be the  payoff  in the game $\Gamma_\lambda$, induced by a pair of stationary strategies $(x,y) \in X \times Y$. It satisfies 
\begin{equation}
r_{\lambda} (x,y) =  \lambda g(x,y) + (1- \lambda) [  (1  - p^*(x,y)) r_{\lambda} (x,y) + G^* (x,y)] 
\end{equation}
hence
\begin{equation}\label{rlambda}
r_{\lambda} (x,y) = \frac{ \lambda g(x,y) + (1- \lambda) G^* (x,y)}{\lambda + (1  - \lambda) p^*(x,y)}.
\end{equation}
In particular for any $x_\lambda \in X $ optimal for Player 1  in $\Gamma_\lambda$, one obtains
\begin{equation}\label{m}
v_\lambda \leq  \frac{ \lambda g(x_\lambda, y) + (1- \lambda) G^* (x_\lambda, y)}{\lambda + (1  - \lambda) p^*(x_\lambda, y)}, \quad \forall y \in Y,
\end{equation}
that one can write
\begin{equation}\label{dps}
v_\lambda \leq  \frac{ g(x_\lambda, y) + \frac{ (1- \lambda)}  { \lambda } G^* ( x_\lambda, y)}{ 1 +  \frac{ (1- \lambda)}  { \lambda }  p^*(  x_\lambda,  y)}, \quad \forall y \in Y.
\end{equation}
Let $\overline x \in X$ be an accumulation point of $ \{x_{\lambda_n} \}$ and  given $\varepsilon >0$ let $\overline \lambda$ in the sequence $ \{ {\lambda_n} \}$  such that
$$
|g(\overline x, y) - g(x_{\overline \lambda},y) | \leq \varepsilon, \qquad \forall y \in  Y
$$
(we use the fact that $g$ is uniformly continuous on $X\times Y$) and
$$
|v_{\overline \lambda} -w| \leq \varepsilon.
$$
Then with $\overline a = \frac{ (1- \overline \lambda)  }{\overline  \lambda } $ and    $ \overline x' = x_{\overline \lambda}$,  (\ref{dps}) implies
\begin{equation}\label{dps1}
w - \varepsilon \leq  \frac{ g( \overline x, y) + \overline a G^*  (\overline x', y)}{ 1 + \overline a p^*  (\overline x', y)} + \varepsilon, \quad \forall y \in  Y.
\end{equation}
On the other hand,  going to the limit in (\ref{m}) leads to
\begin{equation}\label{dps2}
w \:  p^*(\overline x, y') \leq  G^*(\overline x , y' ), \qquad \forall y' \in Y.
\end{equation}
We multiply (\ref{dps1}) by the denominator $1 + \overline a  p^*(\overline x ', y)$ and we add to  (\ref{dps2}) multiplied by $b\in \R_+$ to obtain the property:\\
$ \forall \varepsilon > 0, \exists \  \overline {x}, \overline {x}'\in X$ and $\overline a\in \R_+ $ such that
\begin{equation}\label{dps4}
w \leq  \frac{g(\overline x,y)+ \overline a\: G^*(\overline x',y)+b \: G^*(\overline x,y')}{1+\overline a\: p^*(\overline x',y)+b \: p^*(\overline x,y')}+ 2 \varepsilon, \qquad \forall y, y' \in Y, b \in \R_+
\end{equation}
which implies (\ref{dps0}). Note moreover that $\overline x $  is independent of $\varepsilon$,  hence the result.\\

2) Let $(x, x', a) $ be $\varepsilon$-optimal in the game $\A$ and   $\hat x_\lambda:= \displaystyle {\frac{x+\lambda ax'}{1+\lambda a}}$.
Using \eqref{rlambda}
one obtains
$$
r_{\lambda} (\hat x_{\lambda},y) = \frac{ \lambda [g(x,y)  +  \lambda a  g (x',y)] + (1- \lambda) [G^* (x,y) +  \lambda a  G^*(x',y)]}{\lambda  (1 + \lambda a) + (1  - \lambda)( p^*(x,y) + \lambda a  p^*(x',y))}.
$$
Note that
$$
A (x, x', a, y, y, \frac{1- \lambda}{\lambda})= \frac{ {\lambda} g(x,y)+ {\lambda} a\: G^*(x',y)+ (1- {\lambda})  \: G^*(x,y)}{ {\lambda} + {\lambda} a\: p^*(x',y)+ (1- {\lambda})  \: p^*(x,y)}.
$$
Thus
$$
| r_{\lambda} (\hat x_{\lambda},y)  - A (x, x', a, y, y, \frac{1- \lambda}{\lambda}) | \leq 4C \lambda a
$$
where $C$ is a bound on the payoffs. Hence, for any $y \in  Y$
$$
v - r_{\lambda} (\hat x_{\lambda},y)  \leq \varepsilon +4C \lambda a \leq 2\varepsilon
$$
 for $\lambda$ small enough.
\end{proof}

%
%
%
%

$\A$ is an {\it auxiliary limit game} in the sense that:\\
i) There is a map  $ \phi$ from $S  \times (0, 1]$ to $X$ (that associates  to a strategy of player 1  in $\A$ and a discount factor a stationary strategy of player 1  in  $ \Gamma$).\\
ii) There is a map  $ \psi$ from $Y  \times (0, 1]$ to $T$ (that associates  to a  stationary strategy of player 2  in $\Gamma$ and a discount factor a stationary strategy of player 2  in  $ \A$).\\
iii)
$$
 r_\lambda(\phi(\lambda, s), y ) \geq A(s, \psi ( \lambda, y) )  - o(1),\quad  \forall s \in S, \forall y \in Y
$$
iv) A  dual property holds.\\
These properties imply: $\lim v_\lambda$ exists and equals $v(\A)$.

We then recover Corollary 3.2 in Sorin and Vigeral \cite{SV1}, with a new proof that will be useful in the sequel. 
\begin{corollary}\label{coroJota}
\begin{eqnarray}
 v&=&\min_{y\in Y} \max_{x\in X} \med\left(g(x,y); \sup_{x"|p^*(x",y)>0} \left\{\overline g^*(x",y)\right\}; \inf_{y"|p^*(x,y")>0}\left\{\overline g^*(x,y")\right\}\right)\label{eqrida3} \cr
 &=& \max_{x\in X} \min_{y\in Y}\med\left(g(x,y); \sup_{x"|p^*(x",y)>0} \left\{\overline g^*(x",y)\right\}; \inf_{y"|p^*(x,y")>0}\left\{\overline g^*(x,y")\right\}\right)
\end{eqnarray}
where $\med$ is the median of three numbers, and with the usual
convention that $\sup_{x"\in\emptyset}=-\infty$;
$\inf_{y"\in\emptyset}=+\infty$. Moreover 
 if $(x,x',a)$ (resp $(y,y',\varepsilon)$) is $\varepsilon$-optimal in  $\A$ then $x$ (resp. $y$) is $\varepsilon$-optimal in \eqref{eqrida3}.
\end{corollary}
\begin{proof} $ $ \\
For any $\varepsilon>0$ fix a triplet $(y, y'_\varepsilon,b_\varepsilon) \in T$ of the second player  $\varepsilon$-optimal in $\A$, where we can assume that $y$ does not depend on $\varepsilon$ by the previous result.
 Then for any $x" \in X $ such that $p^*(x",y)>0$, one has 
\begin{equation}\label{eqv+epsilon}
v + \varepsilon \geq A(x, x",+\infty,y,y'_\varepsilon,b_\varepsilon)=  \overline g^*(x",y)
\end{equation}
thus $v + \varepsilon \geq  h^+(y) := \sup_{x"|p^*(x",y)>0} \left\{\overline g^*(x", y)\right\}$. 
Denote similarly $ h^- (x) =  \inf_{y"|p^*(x,y")>0}\left\{\overline g^*(x,y')\right\}$.\\
On the other hand, for any $x$ 
\[
v+ \varepsilon \geq A(x,x",0, y, y'_\varepsilon, b_\varepsilon)=\frac{g(x, y)+b_\varepsilon \: G^*(x, y'_\varepsilon)}{1+b_\varepsilon \: p^*(x,  y'_\varepsilon)}\]
Now if $p^*(x, y'_\varepsilon) >0$, $\displaystyle{\frac{g(x,  y)+b_\varepsilon \: G^*(x, y'_\varepsilon)}{1+b_\varepsilon \: p^*(x,y'_\varepsilon)}}\geq \min \{ g(x,y), 
 \overline g^*(x,y'_\varepsilon)\} $,  hence in all cases
 \[
v+ \varepsilon \geq A(x,x",0,y, y'_\varepsilon,b_\varepsilon)\geq \min\{g(x, y), h^- (x)\}.
\]
Thus  for  any $x \in X$
\begin{equation}\label{eqmedepsilon}
v + \varepsilon \geq \med\left(g(x,y); h^+ (y); h^- (x)\right)
\end{equation}
Letting $\varepsilon$ go to 0 
%
and using  the dual inequality establish the results.
%
%
%
\end{proof}  
 

\subsubsection{Further properties of optimal strategies} 
$ $ 

We establish here more precise results   concerning the  decomposition of the payoff induced by  $\varepsilon$-optimal strategies  in the game $\A$ .
\begin{proposition}\label{propabc}
Let $(x,x',a)$ and $ (y,y',b)$ be  $\varepsilon$-optimal in the game $\A$.
\begin{enumerate}
\item If $ p^* (x, y) > 0$ then $| \overline g^* (x,y)-v|\leq \varepsilon$
\item $|g(x,y)-v|\leq 2(1+ap^*(x',y)+ bp^*(x,y'))\varepsilon$ 
\item If $ap^*(x',y)+bp^*(x,y')>0$ then $\left| \displaystyle{\frac{aG^*(x',y)+bG^*(x,y')}{ap^*(x',y)+bp^*(x,y')}}-v\right|\leq 3 \displaystyle{\frac{1+ap^*(x',y)+ bp^*(x,y')}{ap^*(x',y)+ bp^*(x,y')}}\varepsilon $.
\end{enumerate}
\end{proposition}

\begin{proof}
a) This is exactly equation \eqref{eqv+epsilon} and its dual.

b) From $v+\varepsilon\geq A(x,x',0,y,y',b)$ we get
\[
\frac{g(x,y)+bG^*(x,y')}{1+bp^*(x,y')}\leq v+\varepsilon.
\]
On the other hand, $v-\varepsilon \leq A(x,x',a,y,y',+\infty)$ hence $G^*(x,y')\geq (v-\varepsilon) p^*(x,y')$. 
Combining both inequalities yields
\begin{eqnarray}
\nonumber g(x,y)&\leq& (v+\varepsilon)(1+bp^*(x,y'))-b(v-\varepsilon) p^*(x,y')\\
\nonumber &\leq& v+\varepsilon (1+2 bp^*(x,y'))\\
\label{eqg<v}&\leq&v+2\varepsilon(1+ap^*(x',y)+ bp^*(x,y'))
\end{eqnarray}
and the dual inequality is similar.

c) Since $A(x,x',a,y,y',b)\geq v-\varepsilon$, one has 
\begin{eqnarray*}
(v-\varepsilon)(1+ap^*(x',y)+ bp^*(x,y'))&\leq& g(x,y)+aG^*(x',y)+bG^*(x,y')\\
&\leq& v+2\varepsilon(1+ap^*(x',y)+ bp^*(x,y'))+aG^*(x',y)+bG^*(x,y')  \text{ by \eqref{eqg<v}}
\end{eqnarray*}
hence 
\[
v(ap^*(x',y)+ bp^*(x,y'))-aG^*(x',y)-bG^*(x,y')\leq 3\varepsilon(1+ap^*(x',y)+ bp^*(x,y'))
\]
and the dual inequality is similar.
\end{proof}

%
%
%
%
%
%
%
%
%

  \subsection{Asymptotics properties in $\Gamma_\lambda$ } 
  $ $ 
 
%
%
%
%
%
%
%
%
  
  
 Since the game is absorbing,  we write simply $Q^{x,y}_\lambda(t)$ for $Q^{\omega_0,x,y}_\lambda(t)(\omega_0)$, where $\omega_0$  is the nonabsorbing state.
  
  \begin{lem}\label{lemQgamma}
 Let $x_\lambda$ and $y_\lambda$ be two families of (non necessarily optimal) stationary strategies of Player 1 and Player 2 respectively. Assume that $\displaystyle{\frac{p^*(x_\lambda,y_\lambda)}{\lambda}}$ converges to some $\gamma$ in $[0,+\infty]$  as $\lambda$ goes to 0. Then 
$Q^{x_\lambda,y_\lambda}_\lambda(t)$ converges uniformly in $t$ to
$\displaystyle{\frac{1-(1-t)^{1+\gamma}}{1+\gamma}}$ as $\lambda$ goes to 0,
 with the natural convention that $\displaystyle{\frac{1-(1-t)^{1+\gamma}}{1+\gamma}}=0$ for $\gamma=+\infty$.
  \end{lem}
  
  \begin{proof}
  By definition \ref{dQ}, for any $\lambda$ and $t_n= \lambda\sum_{i=1}^n (1-\lambda)^{i-1}=1-(1-\lambda)^n$,
  \begin{eqnarray*}
  Q^{x_\lambda,y_\lambda}_\lambda(t_n)&= &\lambda\sum_{i=1}^n (1-\lambda)^{i-1} (1-p^*(x_\lambda,y_\lambda))^{i-1}\\
  &=&\frac{1-((1-\lambda) (1-p^*(x_\lambda,y_\lambda)))^n}{1+\frac{p^*(x_\lambda,y_\lambda)}{\lambda}-p^*(x_\lambda,y_\lambda)}
  \end{eqnarray*}
  with linear interpolation between these dates.
  
Remark first that this implies that $Q^{x_\lambda,y_\lambda}_\lambda(t)\leq [{1+\frac{p^*(x_\lambda,y_\lambda)}{\lambda}}]^{-1}$ for all $t$ and $\lambda$, which gives at the limit the desired result if $\gamma=+\infty$. 

Assume now that $\gamma\in[0,+\infty[$, and thus that $p^*(x_\lambda,y_\lambda)$ tends to 0 as $\lambda$ goes to 0. Fix $t$ and $\lambda$, and let $n$ be the integer part of $\frac{\ln(1-t)}{\ln(1-\lambda)}$ so that $t_n\leq t \leq t_{n+1}$. Since $Q^{x_\lambda,y_\lambda}_\lambda(t_n)$ is decreasing in $n$,

\begin{eqnarray}
\nonumber Q^{x_\lambda,y_\lambda}_\lambda(t)&\leq&Q^{x_\lambda,y_\lambda}_\lambda(t_n)\\
\nonumber&=&\frac{1-((1-\lambda) (1-p^*(x_\lambda,y_\lambda)))^n}{1+\frac{p^*(x_\lambda,y_\lambda)}{\lambda}-p^*(x_\lambda,y_\lambda)}\\
\nonumber&\leq&\frac{1-((1-\lambda) (1-p^*(x_\lambda,y_\lambda)))^{\frac{\ln(1-t)}{\ln(1-\lambda)}-1}}{1+\frac{p^*(x_\lambda,y_\lambda)}{\lambda}-p^*(x_\lambda,y_\lambda)}\\
\label{eqQleq}&=&\frac{1-(1-t)^{1+\frac{\ln(1-p^*(x_\lambda,y_\lambda))}{\ln(1-\lambda)}-\ln(1-\lambda)}}{1+\frac{p^*(x_\lambda,y_\lambda)}{\lambda}-p^*(x_\lambda,y_\lambda)}
\end{eqnarray}

  Similarly,
  \begin{eqnarray}
\nonumber Q^{x_\lambda,y_\lambda}_\lambda(t)&\geq&Q^{x_\lambda,y_\lambda}_\lambda(t_{n+1})\\
\nonumber&=&\frac{1-((1-\lambda) (1-p^*(x_\lambda,y_\lambda)))^{n+1}}{1+\frac{p^*(x_\lambda,y_\lambda)}{\lambda}-p^*(x_\lambda,y_\lambda)}\\
\nonumber&\geq&\frac{1-((1-\lambda) (1-p^*(x_\lambda,y_\lambda)))^{\frac{\ln(1-t)}{\ln(1-\lambda)}+1}}{1+\frac{p^*(x_\lambda,y_\lambda)}{\lambda}-p^*(x_\lambda,y_\lambda)}\\
\label{eqQgeq}&=&\frac{1-(1-t)^{1+\frac{\ln(1-p^*(x_\lambda,y_\lambda))}{\ln(1-\lambda)}+\ln(1-\lambda)}}{1+\frac{p^*(x_\lambda,y_\lambda)}{\lambda}-p^*(x_\lambda,y_\lambda)}
\end{eqnarray}

Letting $\lambda$ go to 0 in \eqref{eqQleq} and \eqref{eqQgeq}  yields the result.
  \end{proof}
  
For any $(x,x',a)\in X^2\times\mathbb{R}^+$  and $(y,y',b)\in Y^2\times\mathbb{R}^+$ define 
\[
\gamma(x,x',a,y,y',b)=\begin{cases} +\infty & \text{ if $p^*(x,y)>0$}\\
 ap^*(x',y)+bp^*(x,y')&\text{ if $p^*(x,y)=0$}
 \end{cases}
  \]
  
  An immediate consequence of the previous lemma is
  
  \begin{corollary}\label{corQgamma}
Let $(x,x',a)\in X^2\times\mathbb{R}^+$  and $(y,y',b)\in Y^2\times\mathbb{R}^+$ and denote $\hat x_\lambda:=\frac{x+\lambda ax'}{1+\lambda a}$ and $\hat y_\lambda:=\frac{y+\lambda by'}{1+\lambda b}$. Then $Q^{\hat x_\lambda,\hat y_\lambda}_\lambda(t)$  converges uniformly in $t$ to
$\displaystyle{\frac{1-(1-t)^{1+\gamma(x,x',a,y,y',b)}}{1+\gamma(x,x',a,y,y',b)}}$ as $\lambda$ goes to 0.
  \end{corollary}
\begin{proposition}
Any absorbing game has a LOTM $Q(t)=\frac{1-(1-t)^{1+\gamma}}{1+\gamma}$ for some $\gamma\in [0,+\infty]$, and a LOTP $l(t)=tv$.
\end{proposition}
        
\begin{proof}
For every $n$ let  $(x,x'_n,a_n)$ and $(y,y'_n,b_n)$ be $\frac{1}{n}$-optimal strategies for each player in $\A$ (recall that $x$ and $y$ can be chosen independently on $n$). Up to extraction $\gamma_n:=\gamma(x,x'_n,a_n,y,y'_n,b_n)$ converges to some $\gamma$ in $[0,+\infty]$.

Fix $\varepsilon>0$ and let $n\geq\frac{2}{\varepsilon}$ such that 
\begin{equation}\label{eqgamman}
\left|\frac{1-(1-t)^{1+\gamma_n}}{1+\gamma_n}-\frac{1-(1-t)^{1+\gamma}}{1+\gamma}\right|\leq \frac{\varepsilon}{2}
\end{equation}
on [0,1]. By Proposition \ref{LL} the strategies $\hat x^n_\lambda:=\frac{x+\lambda a_nx'_n}{1+\lambda a_n}$ and $\hat y^n_\lambda:=\frac{y+\lambda b_ny'_n}{1+\lambda b_n}$ are $\varepsilon$-optimal in $\Gamma_\lambda$ for $\lambda$ small enough. Corollary \ref{corQgamma} and equation \eqref{eqgamman} imply that 
\begin{equation}\label{eqQabsorb}
\left|Q^{\hat x^n_\lambda,\hat y^n_\lambda}_\lambda(t)-\frac{1-(1-t)^{1+\gamma}}{1+\gamma}\right|\leq {\varepsilon}
\end{equation}
for all $\lambda$ small enough and $t\in[0,1]$. This answers the first part of the Proposition.

Clearly \begin{equation}\label{eqlQabsorb}
l^{\hat x^n_\lambda,\hat y^n_\lambda}_\lambda(t)=Q^{\hat x^n_\lambda,\hat y^n}_\lambda(t) g(\hat x^n_\lambda,\hat y^n_\lambda)+\left(t-Q^{\hat x^n_\lambda,\hat y^n}_\lambda(t)\right) {\overline g}^*(\hat x^n_\lambda,\hat y^n_\lambda). 
\end{equation}

Recall that the payoff function is assumed bounded by 1. Then equations \eqref{eqQabsorb} and \eqref{eqlQabsorb} imply that for $\lambda$ small enough and every $t$,
\[
\left|l^{\hat x^n_\lambda,\hat y^n_\lambda}_\lambda(t)-\frac{1-(1-t)^{1+\gamma}}{1+\gamma}g(\hat x^n_\lambda,\hat y^n_\lambda)-\left(t-\frac{1-(1-t)^{1+\gamma}}{1+\gamma}\right){\overline g}^*(\hat x^n_\lambda,\hat y^n_\lambda) \right|\leq 2\varepsilon
\]

Since $x^n_\lambda$ and $y^n_\lambda$ converge to $x$ and $y$, we then have for $\lambda$ small enough

\begin{equation}\label{eqlabsorb}
\left|l^{\hat x^n_\lambda,\hat y^n_\lambda}_\lambda(t)-\frac{1-(1-t)^{1+\gamma}}{1+\gamma}g(x,y)-\left(t-\frac{1-(1-t)^{1+\gamma}}{1+\gamma}\right){\overline g}^*(\hat x^n_\lambda,\hat y^n_\lambda) \right|\leq 3\varepsilon 
\end{equation}.

We now consider four separate cases. Basically either $\gamma=0$ or $+\infty$ and equation \eqref{eqlabsorb} implies that $l(t)$ is linear, and hence equals $tv$ since $l(1)=v$ by near optimality of the strategies  $\hat x^n_\lambda$ and $\hat y^n_\lambda$ in $\Gamma_\lambda$ ; or $\gamma\in]0,+\infty[$ and then both $g(x,y)$ and ${\overline g}^*(\hat x^n_\lambda,\hat y^n_\lambda)$ are close to $v$ by Proposition \ref{propabc}, which once again implies $l(t)=tv$.

Case 1 : $p^*(x,y)>0$. Then ${\overline g}^*(\hat x^n_\lambda,\hat y^n_\lambda)$ converges to  ${\overline g}^*(x,y)$ as $\lambda$ go to 0, and by Proposition \ref{propabc} a)  $|{\overline g}^*(x,y)-v|\leq \varepsilon$. Since $\gamma=+\infty$ in that case, equation \eqref{eqlabsorb} yields $|l^{\hat x^n_\lambda,\hat y^n_\lambda}_\lambda(t)-tv| \leq 5\varepsilon $ for all $\lambda$ small enough, uniformly in $t$.

 Case 2 : $p^*(x,y)=0$ and $\gamma=0$, hence $\gamma_n\leq 1$ (up to chosing a larger $n$). Then Proposition \ref{propabc} b) implies $|{g}(x,y)-v|\leq 4\varepsilon$, and  equation \eqref{eqlabsorb} yields $|l^{\hat x^n_\lambda,\hat y^n_\lambda}_\lambda(t)-tv| \leq 7\varepsilon $ for all $\lambda$ small enough, uniformly in $t$.
 
 Case 3 : $p^*(x,y)=0$ and $\gamma\in]0,+\infty[$, hence $\frac{\gamma}{2}\leq\gamma_n\leq 1+\gamma$ (up to chosing a larger $n$). Then Proposition \ref{propabc} b) implies $|{g}(x,y)-v|\leq 2(2+\gamma)\varepsilon$. Moreover, $p^*(x,y)=0$ implies that ${\overline g}^*(\hat x^n_\lambda,\hat y^n_\lambda)$ converges to $\displaystyle{\frac{aG^*(x'_n,y)+bG^*(x,y'_n)}{ap^*(x'_n,y)+bp^*(x,y'_n)}}$ as $\lambda$ go to 0, and by Proposition \ref{propabc} c) $\left|\displaystyle{\frac{aG^*(x'_n,y)+bG^*(x,y'_n)}{ap^*(x'_n,y)+bp^*(x,y'_n)}}-v\right|\leq 3\frac{1+\gamma/2}{\gamma/2}\varepsilon$.
 
Hence equation \eqref{eqlabsorb} yields $|l^{\hat x^n_\lambda,\hat y^n_\lambda}_\lambda(t)-tv| \leq (7+2\gamma+ 3\frac{1+\gamma/2}{\gamma/2})\varepsilon $ for all $\lambda$ small enough, uniformly in $t$. 
 
Case 4 : $p^*(x,y)=0$ and $\gamma=+\infty$, hence $\gamma_n\geq 1$ (up to chosing a larger $n$). Then ${\overline g}^*(\hat x^n_\lambda,\hat y^n_\lambda)$ converges to $\displaystyle{\frac{aG^*(x'_n,y)+bG^*(x,y'_n)}{ap^*(x'_n,y)+bp^*(x,y'_n)}}$ as $\lambda$ go to 0, and by Proposition \ref{propabc} c) $\left|\displaystyle{\frac{aG^*(x'_n,y)+bG^*(x,y'_n)}{ap^*(x'_n,y)+bp^*(x,y'_n)}}-v\right|\leq 6\varepsilon$. Thus equation \eqref{eqlabsorb} yields $|l^{\hat x^n_\lambda,\hat y^n_\lambda}_\lambda(t)-tv| \leq 9\varepsilon $ for all $\lambda$ small enough, uniformly in $t$.

As claimed, in every case we see that $l(t)=tv$.
\end{proof}

\begin{rem}
\end{rem}
Recall that $Q(\cdot)$ and $l(\cdot)$ represent expected  \emph{cumulated} occupation measure and payoff. By deriving these quantities with respect to $t$ we get that the asymptotic probability  $q(t)$ of  still being in the non absorbing  state at time $t$ is $(1-t)^\gamma$, and that the current asymptotic payoff is $v$ at any  time.

\begin{rem}
\end{rem}
Let us give a simple heuristic behind the form  $(1-t)^\gamma$ for  $q(t)$. Assuming that this quantity is well defined and smooth, note that at time $t$ the remaining game has a length $1-t$ and weight $ 1 - q(t)$ hence by renormalization (see figure below)
 
  $$
 \frac{q'(t)}{1-q(t)} (1-t) = q'(0)
 $$
so that
 $$
 - \frac{q'(t)}{1- q(t)} = - \frac{q'(0)}{ 1- t}
 $$
 which leads, with $q(0) = 0$ to
  $q(t) = 1 - (1-t) ^\gamma$ for some $\gamma$.

 
 \scalebox{1} 
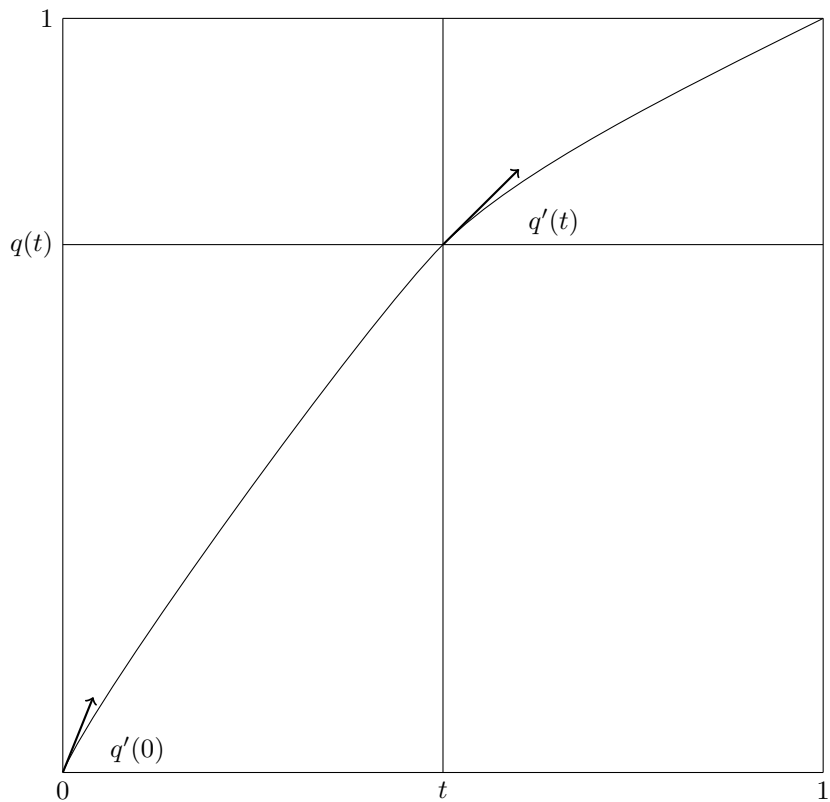
{ \begin{tikzpicture}
\draw (0,10) - - (10,10);
\draw (0,7) - - (10,7);
\draw (0,0) - - (10,0);
\draw (0,0) - - (0,10);
\draw (5,0) - - (5,10);
\draw (10,0) - - (10,10);
\draw (0,0) node [below] {$0$};
\draw (5,0) node [below] {$t$};
\draw (10,0) node [below] {$1$};
\draw (0,7) node [left] {$q(t)$};
\draw (0,10) node [left] {$1$};
\draw (0.5,0) node [above right] {$q'(0)$};
\draw (6,7) node [above right] {$q'(t)$};
\draw [->] [thick] (0,0) - - (0.4,1);
\draw [->] [thick]  (5, 7) - - (6,8);
\draw (0,0) .. controls +(0.4,1) and +(-1,-1) .. (5,7);
\draw (5,7) .. controls +(1,1) and +(-2,-1) .. (10,10);
\end{tikzpicture}}

Let us illustrate now  the four cases in the preceding proof by giving examples.
\begin{example}
\end{example}
Consider the absorbing game
$$
\begin{tabular}{rcc}
&L&R\\  \cline{2-3} T&\multicolumn{1}{|c|}{$1^*$}&\multicolumn{1}{|c|}{$0$}\\
\cline{2-3}
B&\multicolumn{1}{|c|}{$0$}&\multicolumn{1}{|c|}{$1^*$}\\
\cline{2-3}
\end{tabular}
$$
with asymptotic value 1.
Let $x=1/2T+1/2 B$. Then $(x,x,n)$ is $1/n$-optimal in $\A$ for Player 1, while any $(y,y',b)$ is optimal for Player 2. Since $p^*(x,y)>0$ for all $y$ case 1occurs  for any choice of $(y,y',b)$, hence the corresponding $\gamma$ is $+\infty$ and $Q(t)=0$ for all $t$.

Notice that in $\Gamma_\lambda$ the only optimal stationary strategy is $(1/2,1/2)$ for each player, leading to the same asymptotic trajectory $Q(.)=0$. I

\begin{example}
\end{example}
Consider the absorbing game
$$
\begin{tabular}{rcc}
&L&R\\  \cline{2-3} T&\multicolumn{1}{|c|}{$1^*$}&\multicolumn{1}{|c|}{$0$}\\
\cline{2-3}
B&\multicolumn{1}{|c|}{$0$}&\multicolumn{1}{|c|}{$1$}\\
\cline{2-3}
\end{tabular}
$$
with asymptotic value 1.
Then $(B,T,n)$ is $1/n$-optimal in $\A$, while any $(y,y',b)$ is optimal. The associated $\gamma$ is $ny(L)$, hence either $y(L)=0$ and case 2 holds with $\gamma=0$ and $Q(t)=t$, or $y(L)>0$ and  case 4  occurs with $\gamma=+\infty$ and $Q(t)=0$. 

Notice that in $\Gamma_\lambda$ the only optimal stationary strategy is $x_\lambda=y_\lambda=(\frac{\sqrt{\lambda}}{1+\sqrt{\lambda}},\frac{1}{1+\sqrt{\lambda}})$ for each player. Since $p^*(x_\lambda, y_\lambda)=\frac{\lambda}{(1+\sqrt{\lambda})^2}\sim \lambda$, Lemma \ref{lemQgamma} implies that the asymptotic trajectory associated to optimal strategies is $Q(t)=t-\frac{t^2}{2}$. Moreover for any $\gamma\geq0$ the strategy of Player 2 $z_\lambda=(\frac{\gamma\sqrt{\lambda}}{1+\sqrt{\lambda}},1-\frac{\gamma\sqrt{\lambda}}{1+\sqrt{\lambda}})$ is $\varepsilon$-optimal in $\Gamma_\lambda$ for $\lambda$ small enough, and $p^*(x_\lambda, z_\lambda)\sim \gamma \lambda$ hence any $Q(t)$ of the form $\frac{1-(1-t)^{1+\gamma}}{1+\gamma}$ is an asymptotic behavior.

\begin{example}
\end{example}
Consider the Big Match
$$
\begin{tabular}{rcc}
&L&R\\  \cline{2-3} T&\multicolumn{1}{|c|}{$1^*$}&\multicolumn{1}{|c|}{$0^*$}\\
\cline{2-3}
B&\multicolumn{1}{|c|}{$0$}&\multicolumn{1}{|c|}{$1$}\\
\cline{2-3}
\end{tabular}
$$
with asymptotic value 1/2. Let $y=1/2L+1/2R$.
Then $(B,T,1)$ and $(y,y,0)$ are optimal in $\A$, with $\gamma=1$. Hence  case 3 holds, and the corresponding $Q(t)$ is $t-\frac{t^2}{2}$. The optimal strategies in $\Gamma_\lambda$ are $(\frac{\lambda}{1+\lambda},\frac{1}{1+\lambda})$ and $(1/2,1/2)$ respectively, leading to the same $Q(t)=t -\frac{t^2}{2}$. 

%
%

%
%
%

 \subsection{Finite case } $ $ \\
 We now prove that when the game is finite, the limit payoff  trajectory is linear  for every couple of near optimal stationary strategies, not only those given by 
 Proposition \ref{LL}. That is, $l(t)=tv$ is a strong limit behavior for the cumulated payoff.
 
 \begin{proposition}\label{Propfinite}
 Let $\Gamma$ be a finite absorbing game with asymptotic value $v$,  $x_\lambda$ and $y_\lambda$ families of $\varepsilon(\lambda)$-optimal  stationary  strategies in $\Gamma_\lambda$, with $\varepsilon(\lambda)$ going to 0 as $\lambda$ goes to 0. Then for every $t\in [0,1]$, $l^{{x}_\lambda,{y}_\lambda}_\lambda(t)$ converges to $tv$ as $\lambda$ goes to 0. 
 \end{proposition}
 
 We will use in the proof of this proposition the following elementary lemma given without proof.
 
 \begin{lem}\label{lemelem}
 Let $a,b,c,d$ be real numbers with $c$ and $d$ positive. Then $min(\frac{a}{c},\frac{b}{d})\leq\frac{a+b}{c+d}\leq max(\frac{a}{c},\frac{b}{d})$ with equality if and only if $\frac{a}{c}= \frac{b}{d}$
 \end{lem}
 
     \begin{proof}[Proof of Proposition \ref{Propfinite}]

The result is clear for $t=0$ or $1$, assume by contradiction that it is false for some $t\in ]0,1[$. Hence there is a sequence $\lambda_n$ going to 0 and optimal strategies $x_{\lambda_n}$ and $y_{\lambda_n}$ such that $l^{{x}_{\lambda_n},{y}_{\lambda_n}}_{\lambda_n}(t)$ converges to $tw$ with $w\neq v$. Up to extraction of subsequences, $x_{\lambda_n}$ and $y_{\lambda_n}$ converge to $x$  and $y$ respectively. Also up to extraction, all the following limits exist in $[0,+\infty]$: $\alpha(i)=\lim_{n\rightarrow\infty} \frac{x_{\lambda_n}(i)}{\lambda_n}$, $\beta(j)=\lim_{n\rightarrow\infty} \frac{y_{\lambda_n}(j)}{\lambda_n}$, and $\gamma=\lim_{n\rightarrow\infty} \frac{p^*(x_{\lambda_n},y_{\lambda_n})}{\lambda_n}$. If $\gamma\neq 0$ (and hence $p^*(x_{\lambda_n},y_{\lambda_n})>0$ for $n$ large enough), denote $\overline g^*(x,y):=\lim\frac{G^* (x_{\lambda_n},y_{\lambda_n})}{p^* (x_{\lambda_n},y_{\lambda_n})}$, which also exists up to extraction.
 
 Recall formula (\ref{rlambda}):
 \begin{equation} \label{eqvlambdan}
r_{\lambda_n}(x_{\lambda_n},y_{\lambda_n})= \frac{ \lambda_n g(x_{\lambda_n},y_{\lambda_n}) + (1- \lambda_n) G^* (x_{\lambda_n},y_{\lambda_n}) }{\lambda_n + (1- \lambda_n) p^* (x_{\lambda_n},y_{\lambda_n})}
 \end{equation} 
  and since  $x_\lambda$ and $y_\lambda$ are families of $\varepsilon(\lambda)$-optimal strategies in $\Gamma_\lambda$, $r_{\lambda_n}(x_{\lambda_n},y_{\lambda_n})$ converges to $v$ as $n$ tends to infinity.
  
  Recall that by Lemma \ref{lemQgamma}, at the limit
  
  \[
l(t)=\frac{1-(1-t)^{1+\gamma}}{1+\gamma} g(x,y)+\left(t-\frac{1-(1-t)^{1+\gamma}}{1+\gamma}\right) {\overline g}^*(x,y).   
  \]
  
  We first claim that $\gamma\in]0,+\infty[$.\\ 
 If $\gamma=0$,  $l(t)=tg(x,y)$, and by near optimality of $x_\lambda$ and $y_\lambda$ $v=l(1)=g(x,y)$, hence $l(t)=tv$ a contradiction.  \\
  If $\gamma=+\infty$, $l(t)=t \overline g^*(x,y)$, and by near optimality of $x_\lambda$ and $y_\lambda$ $v=l(1)=g^*(x,y)$, hence $l(t)=tv$ a contradiction.  
  
 Hence $\gamma\in]0,+\infty[$, 
   and $w$ is a nontrivial convex combination of $g(x,y)$ and $\overline g^*(x,y)$. Since $v=l(1)$ is also a convex combination of $g(x,y)$ and $\overline g^*(x,y)$, the assumption that $v\neq w$ implies $g(x,y) \neq\overline g^*(x,y)$. Assume without loss of generality  $g(x,y)<v< \overline g^*(x,y)$.
  
  We class the actions $i$ of Player 1 in 4 categories $I_1$ to $I_4$: 
  \begin{itemize}
  \item $i\in I_1$ if $x(i)>0$,
    \item $i\in I_2$ if $x(i)=0$ and $\alpha(i)=+\infty$,  
    \item $i\in I_3$ if $x(i)=0$ and $\alpha(i)\in]0,+\infty[$,
      \item $i\in I_4$ if $x(i)=0$ and $\alpha(i)=0$.
     \end{itemize}
     Hence actions of category 1 are of order 1, actions of category 3 are of order $\lambda$, actions of category 4 are of order $o(\lambda)$, and actions of category 2 are played with probability going to 0 but large with respect to  $\lambda$. Define categories $J_1$ to $J_4$ of player 2 in a similar way.
     By definition,
     \begin{eqnarray} \label{eqpetoile}
     \frac{p^*(x_{\lambda_n},y_{\lambda_n})}{\lambda_n}&=&\sum_{I\times J} \frac{p^*(i,j) x_{\lambda_n}(i) y_{\lambda_n}(j)}{\lambda_n}
     \end{eqnarray}     
    Recall that the left hand side converges to $\gamma  \in]0,+\infty[$, hence up to extraction $\frac{p^*(i,j)x_{\lambda_n}(i){y_{\lambda_n}}(j)}{\lambda_n}$ converge in $[0,+\infty[$ for any $i$ and $j$, denote by  $\delta_{ij}$ the limit. If $i$ and $j$ are of category $k$ and $l$ with $k+l>4$ then $\delta_{ij}=0$. If $k+l<4$ then $\frac{x_{\lambda_n}(i){y_{\lambda_n}}(j)}{\lambda_n}$ diverges to $+\infty$ which implies that $p^*(i,j)=0=\delta_{ij}$.

%

Hence going to the limit in (\ref{eqpetoile}) we get that $\gamma=\gamma_{1,3}+\gamma_{2,2}+\gamma_{3,1}$ where
  where $\gamma_{1,3}:= \sum_{I_1\times J_3}p^*(i,j) x(i)\beta(j)$,  $\gamma_{2,2}:= \sum_{I_2\times J_2} \delta_{ij}$, and $\gamma_{3,1}:= \sum_{I_3\times J_1}p^*(i,j)  \alpha(i)y(j)$.
Recall that $\gamma>0$ thus at least one of $\gamma_{1,3}$, $\gamma_{2,2}$ or $\gamma_{3,1}$ is positive as well. 
    
Similarly, passing to the limit in the definition of $G^*(x_{\lambda_n},y_{\lambda_n}):= \sum_{I\times J} G^*(i,j) x_{\lambda_n}(i) y_{\lambda_n}(j)$
  yields $\mu:=\lim \frac{G_{kl}^*(x_{\lambda_n},y_{\lambda_n}) }{\lambda_n}=\mu_{1,3}+\mu_{2,2}+\mu_{3,1}$  where $\mu_{1,3}:= \sum_{I_1\times J_3} x(i)\beta(j) G^*(i,j)$,  $\mu_{2,2}:= \sum_{I_2\times J_2} \delta_{ij}\overline g^*(i,j)$, and $\mu_{3,1}:= \sum_{I_3\times J_1} \alpha(i)y(j)G^*(i,j)$. Note that if $\gamma_{k,4-k}=0$ for some $k$ then $\mu_{k,4-k}=0$ as well.
  
  Finally, going to the limit in equation (\ref{eqvlambdan}) yields
  \[
  v=\frac{ g(x,y)+\mu_{1,3}+\mu_{2,2}+\mu_{3,1}}{1+\gamma_{1,3}+\gamma_{2,2}+\gamma_{3,1}}
  \]
  and similarly going to the limit in the definition of $\overline g^*(x,y):=\lim\frac{G^* (x_{\lambda_n},y_{\lambda_n})}{p^* (x_{\lambda_n},y_{\lambda_n})}$ yields
  
    \[
\overline g^*(x,y)=\frac{ \mu_{1,3}+\mu_{2,2}+\mu_{3,1}}{\gamma_{1,3}+\gamma_{2,2}+\gamma_{3,1}}
  \]
  
  Recall that we assumed $\overline g^*(x,y)>v$. By Lemma \ref{lemelem}, there exists $k$ such that $\gamma_{k,4-k}>0$ and $\frac{\mu_{k,4-k}}{\gamma_{k,4-k}}\geq \overline g^*(x,y)>v$.
  
  Assume first that $k=1$. Consider now the following strategy $y'_{\lambda_n}$ : $y'_{\lambda_n}(j)=0$ for $j\in J_3$, and $y'_{\lambda_n}(j)=y_{\lambda_n}(j)$ for all other $j$ except for an arbitrary $j_0\in J_1$ for which $y'_{\lambda_n}(j_0)=y_{\lambda_n}(j_0)+\sum_{j\in J_3} y_{\lambda_n}(j)$. The only effect of this deviation is that now $\gamma'_{1,3}=\mu'_{1,3}=0$. Hence
   \[
  \lim_{n\rightarrow\infty} r_{\lambda_n}(x_{\lambda_n},y'_{\lambda_n})= \frac{ g(x,y)+\mu_{2,2}+\mu_{3,1}}{1+\gamma_{2,2}+\gamma_{3,1}}
  \]
  which is strictly less than $v$ by Lemma  \ref{lemelem} since $v<\frac{\mu_{1,3}}{\gamma_{1,3}}$ ; this contradicts the $\epsilon(\lambda_n)$-optimality of $x_{\lambda_n}$.
  
  Assume next that $k=2$. Consider now the following strategy $y'_{\lambda_n}$ : $y'_{\lambda_n}(j)=0$ for $j\in J_2$, and $y'_{\lambda_n}(j)=y_{\lambda_n}(j)$ for all other $j$ except for an arbitrary $j_0\in J_1$ for which $y'_{\lambda_n}(j_0)=y_{\lambda_n}(j_0)+\sum_{j\in J_2} y_{\lambda_n}(j)$. The only effect of this deviation is that now $\gamma'_{2,2}=\mu'_{2,2}=0$. Hence
   \[
  \lim_{n\rightarrow\infty} r_{\lambda_n}(x_{\lambda_n},y'_{\lambda_n})= \frac{ g(x,y)+\mu_{1,3}+\mu_{3,1}}{1+\gamma_{2,2}+\gamma_{3,1}}
  \]
  which is strictly less than $v$ by Lemma  \ref{lemelem} since $v<\frac{\mu_{2,2}}{\gamma_{2,2}}$ ; this again contradicts the  $\epsilon(\lambda_n)$-optimality of $x_{\lambda_n}$.
  
  Finally assume that $k=3$. Consider now the following strategy $x'_{\lambda_n}$ : $x_{\lambda_n}(i)=2x_{\lambda_n}(i)$  for $i\in I_3$ and $x'_{\lambda_n}(i)=x_{\lambda_n}(i)$ for all other $i$ except  for an arbitrary $i_0\in I_1$ for which $x'_{\lambda_n}(i_0)=x_{\lambda_n}(i_0)-\sum_{j\in J_3} x_{\lambda_n}(j)$ (which is nonnegative for $n$ large enough). The only effect of this deviation is that now $\gamma'_{3,1}=2\gamma_{3,1}$ and $\mu'_{3,1}=2\mu_{3,1}$. Hence 
  
   \[
  \lim_{n\rightarrow\infty} r_{\lambda_n}(x'_{\lambda_n},y_{\lambda_n})=\frac{ g(x,y)+\mu_{1,3}+\mu_{2,2}+2\mu_{3,1}}{1+\gamma_{1,3}+\gamma_{2,2}+2\gamma_{3,1}}  \]
    which is strictly more than $v$ by Lemma  \ref{lemelem} since $v<\frac{\mu_{3,1}}{\gamma_{3,1}}$; this contradicts the  $\varepsilon(\lambda_n)$-optimality of $y_{\lambda_n}$.
 \end{proof}

\section{Stochastic   finite games}

\subsection{Non algebraic limit trajectories}$ $ \\
Consider the following zero-sum stochastic game with two non
absorbing states and two actions for each player. In the first state
$s_1$ (which is the starting state) the payoff and transitions are
as follows:
$$
\begin{tabular}{rcc}
&L&R\\  \cline{2-3} U&\multicolumn{1}{|c|}{1*}&\multicolumn{1}{|c|}{0+}\\
\cline{2-3}
D&\multicolumn{1}{|c|}{0}&\multicolumn{1}{|c|}{1}\\
\cline{2-3}
\end{tabular}
$$

\noindent where $*$ denotes absorption and $+$ that there is a
deterministic transition to state 2. Starting from the second state $s_2$ the
game is a linear variation of the Big Match:
$$
\begin{tabular}{rcc}
&L&R\\  \cline{2-3} U&\multicolumn{1}{|c|}{1*}&\multicolumn{1}{|c|}{-1*}\\
\cline{2-3}
D&\multicolumn{1}{|c|}{-1}&\multicolumn{1}{|c|}{1}\\
\cline{2-3}
\end{tabular}
$$
Since $v_\lambda(s_2)=0$ for all $\lambda$ and since there is no
return once the play has entered state $s_2$, it implies that
the optimal play in state $s_1$ is the same than in the Big Match, in which $\gamma=1$. So in
both states the optimal strategies in $\Gamma_\lambda$ are
$D+\lambda U$ for Player 1 and $1/2L+1/2R$ for Player 2. By a
scaling of time, the preceding section tells us that at the limit
game, the probability of being in state 2 at time $t$, given that
there were transition from $s_1$ to $s_2$ at time $z$, is
$\frac{1-t}{1-z}$. Since (also from the preceding section) the time
of transition from $s_1$ to another state (which is $s_2$ with probability $1/2$) has a uniform law on $[0,1]$, the
probability of being in $s_2$ at time $t$ is
\[
p(t)=\frac{1}{2}\int_0^t \frac{1-t}{1-z} dz=-\frac{(1-t)\ln(1-t)}{2}.
\]
Notice that this not an algebraic function of $t$ as it was always
the case in the preceding section. Similarly the probability of
absorption before time $t$ is $1-(1-t)(1-\frac{\ln(1-t)}{2})$ and is also non
algebraic.

\subsection{No $\varepsilon-$optimal strategies of the form $x+a\lambda x'$ } $ $\\
 In the following game with two non absorbing states the payoff is always 1 in state $a$ and  -1 in state $b$ with the following deterministic transitions
$$ 
\mbox {   state $a$}
\qquad  \begin{tabular}{rcc}
&L&R\\  \cline{2-3}
U&\multicolumn{1}{|c|}{a}&\multicolumn{1}{|c|}{b}\\
\cline{2-3}
D&\multicolumn{1}{|c|}{b}&\multicolumn{1}{|c|}{1*}\\
\cline{2-3}
\end{tabular}
$$

$$ \mbox {   state $b$}\qquad  \begin{tabular}{rcc}
&L&R\\  \cline{2-3} U&\multicolumn{1}{|c|}{b}&\multicolumn{1}{|c|}{a}\\
\cline{2-3}
D&\multicolumn{1}{|c|}{a}&\multicolumn{1}{|c|}{-1*}\\
\cline{2-3}
\end{tabular}
$$
It is easy to see that the asymptotic value is 0 and that optimal strategies in the  $\lambda$-discounted game   put a weight  $ \sim\sqrt \lambda$ on $D$ and $R$ in both states, hence the absorbing probability is of the order of $\lambda $ per stage  in each state.\\
We show that strategies of the form $(x_a +  C_a \lambda x'_a$, $x_b +  C_b \lambda x'_b)$
cannot guarantee more than -1 to player 1, as $\lambda $ goes to $0$.\\
- If $x_a(D) x_b (D) >0$, player 2 plays $L$ in $a$ and $R$ in $b$ inducing an absorbing payoff of -1.\\
From $a$ we reach eventually $b$ were $-1$ has a positive probability.\\ 
- If $x_a(D) = 0, x_b (D) >0$, player plays $R$ in both games  inducing an absorbing payoff of -1.\\
The payoff will be absorbing with high probability in finite time and the   relative probability of $1^*$ vanishes with $\lambda$.\\ 
- If $x_a(D) >0 , x_b (D) =0$, player plays $L$ in both games  inducing a non  absorbing payoff of -1.\\
For $\lambda$ small, most of the time the state is $b$.\\ 
- If $x_a(D) =0 , x_b (D) =0$, player plays $R$ in  game $a$ and $L$ in game $b$.
The event ``absorbing payoff of 1" occurs at stage $n$ if $\omega_n = a$ and $i_n= D$. Hence $\omega_{n-1} = b$ and $i_{n-1} = D$. Now this event ``$i_n= D$ and  $i_{n-1} = D$" has probability  of order $\lambda ^2$. Then the absorbing  component  of  the $\lambda$ discounted payoffs  converges to 0 with $\lambda$.  Moreover the  non absorbing payoff is  mainly $-1$.

\section{An absorbing  game with compact action sets and non linear LOTP}
%
%
%

We consider the following absorbing game with compact actions sets. There are three states, two absorbing $0^*$ and $-1^*$ , and the non absorbing state $\omega$, in which the payoff is 1 whatever the actions taken. The sets of action are $X=Y=\{0\}\cup\{1/n,\ n\in\mathbf{N}^*\}$ with the usual distance. The probabilities of absorption are given by :
\[
\rho(0^*|x,y)=\begin{cases}
0 & \text{if $x=y$}\\
\sqrt{y} & \text{if $x\neq y$}
\end{cases}
\]
and
\[
\rho(-1^*|x,y)=\begin{cases}
y & \text{if $x=y$}\\
0 & \text{if $x\neq y$}
\end{cases}
\]
It is easily checked that both functions $\rho(0^*|\cdot \cdot )$ and $\rho(-1^*|\cdot \cdot )$ are (jointly) continuous.

\begin{proposition}
For any discount factor $\lambda\in ]0,1]$, 0 (resp. 1) is optimal for Player 1 (resp. Player 2) in the $\lambda$-discounted game, and $v_\lambda=\lambda$. \\
The corresponding  payoff trajectory is: $ l(t) = 0$ on $[0,1]$.
\end{proposition}

\begin{proof}
Action 0 of Player 1 ensures that there will never be absorption to state $-1^*$, and thus that the  stage payoff from stage 2 on is nonnegative. Action 1 of Player 2 ensures that there will be absorption with probability 1 at the end of stage 1, and thus that the  stage payoff from stage 2 on is nonpositive. Since the payoff in stage 1 is 1 irrespective of player's actions, the proposition is established.

Notice that the play  under this couple of optimal strategies is simple: there is immediate absorption to $0^*$, and in particular the limit  payoff  trajectory is linear and equals 0 for every time $t$.
\end{proof}

We now prove that there are other $\varepsilon$-optimal strategies, with a different limit payoff trajectory. Denote $\{\lambda\}:=\frac{1}{[1/\lambda]}$ where $[\cdot]$ is the integer part ; hence $\lambda\leq \{\lambda\} < \frac{\lambda}{1-\lambda}$ and $1/\{\lambda\} \in \mathbb{N}^*$ for all $\lambda\in ]0,1]$.

\begin{proposition}
For any discount factor $\lambda\in ]0,1]$, $\{\lambda\}$ is $\lambda$-optimal for Player 1 and $\sqrt{\lambda}$-optimal for Player 2 in the $\lambda$-discounted game.\\
The corresponding payoff trajectory is: $l(t) = t-t^2$.
\end{proposition}

\begin{proof}
 If both players play $\{\lambda\}$, the payoff in the  $\lambda$-discounted game is, according to formula (\ref{rlambda}),
\begin{eqnarray*}
r_{\lambda} (\{\lambda\},\{\lambda\}) &=& \frac{ \lambda  - (1- \lambda)\{\lambda\} }{\lambda + (1  - \lambda) \{\lambda\}}
\end{eqnarray*}
which is nonnegative since  $\{\lambda\} < \frac{\lambda}{1-\lambda}$. On the other hand, since  $\lambda\leq \{\lambda\} $, one gets $r_{\lambda} (\{\lambda\},\{\lambda\})\leq \lambda $. 

If Player 1 plays $\{\lambda\}$ while Player 2 plays $y\neq\{\lambda\}$, there is no absorption to $-1^*$ hence $r_{\lambda}(\{\lambda\},y)\geq 0$.\\
Thus $\{\lambda\}$ is $\lambda$-optimal for Player 1.

If If Player 2 plays $\{\lambda\}$ while Player 1 plays $x\neq\{\lambda\}$, then, according once again to formula (\ref{rlambda}),
\begin{eqnarray*}
r_{\lambda} (x,\{\lambda\}) &=& \frac{ \lambda  }{\lambda + (1  - \lambda) \sqrt{\{\lambda\}}}\\
&\leq& \frac{ \lambda  }{\lambda + (1  - \lambda) \sqrt{\lambda}}\\
&\leq& \sqrt{\lambda}.
\end{eqnarray*}
Thus  $\{\lambda\}$ is $\sqrt{\lambda}$-optimal for Player 2.

Notice that while the limit value is 0 and $(\{\lambda\}, \{\lambda\})$ is a couple of near optimal strategies, along the induced  play the nonabsorbing payoff is 1 and the absorbing payoff is -1. One can compute that the associated $\gamma$ is 1, hence under these strategies $Q(t)=t-\frac{t^2}{2}$. So that  the accumulated limit payoff up to time $t$ is $t-t^2$, which is non linear and positive for every $t\in ]0,1[$.
\end{proof}

Basically the players use  a jointly controlled  procedure either to follow $(\{\lambda\}, \{\lambda\})$ or to get at most (resp. at least) 0. 
\section{Concluding comments}

 A first serie of interesting open questions  is directly related to the results presented here like: \\
extension  of Proposition  12  to general (not stationary) strategies, \\
or more generally  analysis  in the framework of arbitrary   (not discounted) evaluations  and general stochastic games.

%
%
%
It is also natural to consider  other families of repeated games: a first class that is of interest  is games with incomplete information. The natural equivalent of $LOTM$ is in this framework is  the speed at which the information is transmitted  during the game.

\end{document}